\documentclass[12pt]{amsart}
\usepackage{amssymb}
\usepackage{amssymb,amsfonts,amsmath}
\usepackage[all]{xy}

\newtheorem{theorem}[subsection]{Theorem}
\newtheorem{lemma}[subsection]{Lemma}

\newtheorem{corollary}[subsection]{Corollary}

\newtheorem{conjecture}[subsection]{Conjecture}

\theoremstyle{definition} \theoremstyle{remark}

\newcommand{\n}{\nabla}
\newcommand{\de}{\delta}

\newcommand{\ti}{\tilde}
\newcommand{\fr}{\frac}

\newcommand{\g}{\gamma}
\newcommand{\e}{\epsilon}
\newcommand{\p}{\partial}
\newcommand{\s}{\sigma}
\newcommand{\ta}{\tau}
\newcommand{\la}{\lambda}

\begin{document}

\title[Knots in Riemannian manifolds]{Knots in Riemannian  manifolds}
\author{Fuquan Fang}
\thanks{Fuquan Fang was supported by NSF Grant of China \#10671097 and the Capital Normal University}
\address{Institute of Mathematics and interdisciplinary Science, Capital Normal University,
Beijing 100048, P.R. China} \email{fuquan\_fang@yahoo.com}

\author{S\'ergio Mendon\c ca}
\thanks{S\'ergio Mendon\c ca was
supported with a fellowship from CNPq, Brazil}
\address{Departamento de An\'alise, Instituto de Matem\'atica,
Universidade Federal Fluminense, Niter\'oi, RJ, CEP 24020-140,
Brasil} \email{mendonca@mat.uff.br,
sergiomendoncario@yahoo.com.br}

\date{}



\begin{abstract} In this paper we study submanifold with
nonpositive extrinsic curvature in a positively curved manifold. Among other
things we prove that, if $K\subset (S^n, g)$ is a totally geodesic submanifold
in a Riemannian sphere with positive sectional curvature where $n\ge 5$, then $K$ is
homeomorphic to $S^{n-2}$ and the fundamental group of the knot complement $\pi _1(S^n-K)\cong
\Bbb Z$.
\end{abstract}

\maketitle

\section{\bf Introduction}

In \cite{Re} the author constructed nontrivial torus knots in $S^3$
which are totally geodesic with respect to some Riemannian metric
with positive curvature. Inspired by this work, it is interesting to
ask the following

\vskip 2mm

\noindent {\bf Problem 1}: {\it Let $(S^n, g)$ be a Riemannian
sphere with positive sectional curvature and let $i: K \to (S^n, g)$
be a codimension $2$ totally geodesic submanifold. Could $i(K)$ be a
nontrivial knot if $n\ge 2$?}

\vskip 2mm

The problem emerges naturally in the study of transformation group
theory acting on manifolds. The famous Smith conjecture asserts that
if a cyclic group acting on $S^3$ has $1$-dimensional fixed point
set, then the fixed point set must be an unknot. The proof of the
conjecture was finally given in 1979 depended on several major
advances in 3-manifold theory, in particular the work of William
Thurston on hyperbolic structures on 3-manifolds, results by William
Meeks and Shing-Tung Yau on minimal surfaces in 3-manifolds, and
work by Hyman Bass on finitely generated subgroups of $GL(2,\Bbb C)$
(cf. [MB]). The well-known Thurston's conjecture for $3$-orbifold,
proved in [BLP], implies that any such an action is topologically
conjugate to a linear action. Clearly, every fixed point component
of a linear action is a trivial knot and so the latter assertion
implies the Smith conjecture.

However, it is known for a long period that the higher dimensional
($\geq 5$) analog of Smith's conjecture is not true (cf. [Go]). It
is interesting to ask the Riemannian geometric version of the higher
dimensional Smith conjecture:

\noindent {\bf Problem 2}: {\it If $\Bbb Z_p$ acts isometrically on
a positively curved Riemannian sphere $(S^n, g)$ with a codimension
$2$ fixed point set $K$, where $n\ge 5$, then $K$ is a unknot.}

It is easy to see that the above conjecture is equivalent to claim
that the action is topologically conjugate to a linear action.
Therefore, a possibly higher dimensional analog of Thurston's
elliptic orbifold conjecture is the following

{\it If a Lie group $G$ acts isometrically on a positively curved
Riemannian sphere $(S^n, g)$, then the $G$-action is topologically
conjugate to a linear action.}

To attack the above question the first measure is the fundamental
group of the knot complement.  In dimension $3$ it is well-known
that a knot is trivial if and only if its complement has infinite
cyclic fundamental group. The main result in this paper implies that
the higher dimensional Smith conjecture is true at the fundamental
group level, namely, the knot complement $S^n-K$ has infinite cyclic
fundamental group.

The above problems have their analogs in algebraic geometry. It is
an important theme to study the complement of an algebraic variety
in $\Bbb P^n$. The main tool in the studies is the well-known
Zariski connectedness theorem.

 Recently, Wilking \cite{Wi}
obtained a connectedness theorem for totally geodesic submanifolds
in a positively curved closed manifold. In \cite{FMR} the authors
generalized Wilking's Theorem and the Frankel Theorems
(cf.\cite{Fr2}) to a unified connectedness principle (compare
\cite{Lf}, \cite{FL}, \cite{FM}) which applies also to the case of
non totally geodesic submanifolds, e.g., submanifolds with
nonpositive extrinsic curvature.\footnote{An isometric immersion
$f:N\to M$ has nonpositive {\it extrinsic curvature} if the
sectional curvatures $K(X,Y)$ of $N$ do not exceed the corresponding
sectional curvatures $\bar K(X,Y)$ in $M$.} The principle may be
considered as the Riemannian geometric counterparts of the
connectedness principle in algebraic geometry, except the following
Zariski type connectedness conjecture:

\begin{conjecture}
Let $M$ be an $m$-dimensional Riemannian manifold of positive
sectional curvature. If $N$, $H$ are totally geodesic closed
submanifolds of dimensions, respectively, $n$ and $h$, which
intersect transversely, then the homomorphism
$$j_*: \pi _i(N-(N\cap H))\to \pi _i(M-H)$$
induced by inclusion is an isomorphism for $i\le 2n-m$ and an
epimorphism for $i=2n-m+1$.
\end{conjecture}

We remark that, by the transversality theorem, $N\cap H\subset N$ is a submanifold
of codimension $m-h$, and the inclusion induces an isomorphism $\pi _i(N-N\cap H)\cong \pi _i(N)$ if
$i\le m-h-2$ and an epimorphism $\pi
_i(N-N\cap H)\twoheadrightarrow \pi _i(N)$ if $i=m-h-1$.

As a byproduct of this paper we prove the above conjecture at the
level of fundamental group. By the above remark and Frankel's
theorem \cite{Fr2} it is new only if $h=m-2$.

Now let us state our main results:

\begin{theorem}

Let $(S^{n}, g)$ be a Riemannian sphere with positive sectional
curvature. If $K\subset (S^{n}, g)$ is a closed embedded submanifold
of codimension $2$ with nonpositive extrinsic curvature, then $K$ is
homeomorphic to $ S^{n-2}$ and the fundamental group $\pi
_1(S^{n}-K)\cong  \Bbb Z$, provided one of the following conditions
holds:

(1.2.1) $n\ge 14$;

(1.2.2) $n\ge 5$ and $K$ is totally geodesic.
\end{theorem}

\vskip2mm

It is possible to classify a higher dimensional knot $K\subset S^n$
such that $\pi _1(S^n-K)\cong \Bbb Z$ in certain range, e.g., a
theorem of Levine \cite{Le} asserts that $3$-knots in $S^5$ are
determined up to isotopy by the $S$-equivalence classes of their
 Seifert matrices. Following
Levine, we call such a knot a {\it simple knot}.

\vskip 2mm

\begin{corollary}

Let $K\subset (S^5, g)$ be a totally geodesic $3$-manifold in a
Riemannian $5$-sphere with positive sectional curvature. Then $K$ is
a simple knot.
\end{corollary}

\vskip2mm

The following theorem particularly verifies  the first nontrivial case of
Conjecture 1.1:

\begin{theorem}
Let $M$ be an $m$-dimensional closed Riemannian manifold of positive
sectional curvature. Let $H\subset M$ be a codimension $2$
submanifold of non-positive extrinsic curvature, and let $N\subset
M$ be a  submanifold of dimension $n$ which in general position with
$H$ (i.e., intersects $H$ transversely) with non-positive extrinsic
curvature, then the homomorphism
$$j_*:\pi_1(N-(N\cap H))\to \pi_1(M-H)$$
induced by the inclusion is surjective, provided  one of the
following conditions holds:

(1.4.1) $n\ge \fr{3m}4$ and $m\ge 9$;

(1.4.2) $n\ge \frac 12 m$ and $m\ge 5$, provided $N, H$ are both totally geodesic.
\end{theorem}

For a Riemannian manifold $M$ and a given integer $1\le k\le
m-1$, we say that $M$ has positive $k$-Ricci curvature if
$$\sum_{i=1}^kK(v,e_i)>0,$$
where $v\in T_pM$ is any unit tangent vector, $K(v,e_i)$ is the
sectional curvature associated to the plane generated by $v$ and
$e_i$, and $v, e_1,\cdots,e_k$ are orthonormal vectors. This
definition is due to Rovenski (\cite{Ro}). A slightly different definition of
positivity of $k$-Ricci curvature was given previously by Wu (\cite{Wu}).
Observe that $M$ has positive sectional curvature when $k=1$ and
positive Ricci curvature when $k=m-1$.

\vskip 2mm

\noindent {\bf Remark 1.5}: The same conclusion in the above Theorem 1.4 holds
true if we replace the inequalities by either of the following conditions:

(1.5.1) $n\ge   \fr{3m+k-1}4$ and  $m\ge k+8$;

(1.5.2) $n\ge \fr{m+k-1}2$ and $m\ge k+4$, provided both $N, H$ are totally geodesic.

\noindent {\bf Remark 1.6}: The proof to Theorem 1.2 implies easily
the following consequence: Let $M$ be a closed $m$-dimensional
Riemannian manifold and $H$ a simply connected codimension $2$
submanifold with  nonpositive extrinsic curvature. Then $\pi_1(M-H)$
is cyclic if one of the following conditions holds:

(1.6.1) $M$ has positive sectional curvature, $m\ge 8$;

(1.6.2) $M$ has positive $k$-Ricci curvature, $m\ge  k+7$.

\section{\bf Two key lemmas}

Let $M$ be a closed $m$-dimensional Riemannian manifold, and let
$H\subset M$ be a closed submanifold. Following \cite{FMR}, the
asymptotic index of $H\subset M$ is defined by $\nu _H
=\text{min}_{x\in H} \nu (x)$, where $\nu(x)$ is the maximal
dimension of a subspace of $T_xH$ on which the second fundamental
form vanishes (cf. p 188 of [Fl]). Clearly, $H$ is totally geodesic
if and only if $\nu _H=\dim(H)$.

In this section we will prove two key lemmas using variation
theory in the presence of positive curvature.

\begin{lemma} Let $M$ be a closed $m$-dimensional Riemannian manifold
of positive $k$-Ricci curvature and let $H$ be an embedded
submanifold of $M$. Let $V$ be the $\e$-tubular open neighborhood of
$H$ for some small $\e>0$, with closure $\bar
V$ and boundary $\p \bar V$. Let $\gamma:[0,1]\to M$ be a geodesic
such that

(2.1.1) $\g(t)\in M-\bar V$
for $t\in (0,1)$ and $\g(0),\g(1)\in \p \bar V$;

(2.1.2) $\g'(0),\g'(1)\perp\p \bar V$.

If $\nu_H\ge \fr{m+k-1}2$, then there
exists a smooth variation $\g_s$ of $\g=\g_0$ with
$\g_s(0),\g_s(1)\in \p \bar V$ and $\g_s(t)\in M-\bar V$ for all
$t\in(0,1)$ and all $s$, such that the length $L(\g_s)<L(\g)$ if
$s\not=0$.
\end{lemma}
\begin{proof} For $\e$ sufficiently small $\p \bar V$ is contained in the
$2\e$-tubular neighborhood $U$ of $H$. We recall that, for all $x\in
U$, the gradient of the distance function $\rho$ from $H$ satisfies
$\n \rho(x)=\s'(\rho(x))$, where $\s:[0,\rho(x)]\to U$ is a minimal
geodesic from $H$ to $x$ with $\s'(0)\perp H$ and
$|\s'(\rho(x))|=1$. Since $\g'(0),\g'(1)\perp\p \bar V$ and $\p \bar
V=\rho^{-1}(\{\e\})$, we have that $\g'(0)=\la_1\n\rho(\g(0))$ and
$\g'(1)=\la_2\n\rho(\g(1))$. Thus the geodesic $\g$ can be extended
to a geodesic $\ti \g:[-\e,1+\e]\to M$ with $p=\ti\g(-\e)\in
H,q=\ti\g(1+\e)\in H$ and $\ti\g'(-\e),\ti\g'(1+\e)\perp H$.

The parallel transport along $\ti \g$ defines an isometric linear
injective map  $P:T_pH\to T_qM$. Consider $\nu_H$-dimensional linear
subspaces $W\subset T_pH$ and $\widetilde W\subset T_qH$ such that
the second fundamental form $\alpha$ vanishes on $W$ and $\widetilde
W$. Since $\nu_H\ge \fr{m+k-1}2$ and $(P(W)+\widetilde W)\subset
\{\ti \g'(1+\e)\}^\perp$ we have
$$\dim(P(W)\cap \widetilde W)\!=\!\dim(P(W))+\dim(\widetilde W)-\dim(P(W)+\widetilde W)$$
$$\ge m+k-1-\dim(P(W)+\widetilde W)\ge m+k-1-(m-1)=k,$$
hence there exist orthonormal parallel vector fields
$e_1(t),\cdots,e_k(t)$ along $\tilde\g$ with $e_i(-\e)\in W$ and
$e_i(1+\e)\in \widetilde W$ for $i=1,\cdots,k$.

For each $i$ consider a variation $\g_s^i$ of $\tilde\g$ with
$\g_s^i(-\e),\g_s^i(1+\e)\in H$ and  $\frac{\partial
\g^i_s(t)}{\partial s}=e_i(t)$ and let $L_i(s)$ be the length of the
curve $\g^i_s$. The positivity of the $k$-Ricci curvature, the fact
that $\alpha$ vanishes on $W$ and $\widetilde W$, and the second
variation formula imply together that there exists $i$ such that
$L_i''(0)<0$. Thus we get a variation $\g_s^i$, denoted by $\ti\g_s$,
so that
$$\ti\g_0=\ti\g, \ti\g_s(-\e),\ti\g_s(1+\e)\in H\text{ and
} L(\ti\g_s)<L(\ti\g)\text{ if }s\not=0$$ Obviously, there exist
$-\e<t_s<u_s<1+\e$ such that $\ti\g_s(t_s), \ti\g_s(u_s)\in \p \bar
V$ and the image of the restriction $\ti\g_s|_{(t_s,u_s)}$ is
contained in $M-\bar V$. By the transversality of the intersection
between $\g$ and $\p \bar V$ we obtain that $t_s, u_s$ depend
smoothly on $s$. For $s\not=0$ we have:
$$
2\e+L(\g)=L(\ti\g)>
L(\ti\g_s)=L(\ti\g_s|_{[-\e,t_s]})+L(\ti\g_s|_{[t_s,u_s]})+L(\ti\g_s|_{[u_s,1+\e]})$$
$$\ge 2\e+L(\ti\g_s|_{[t_s,u_s]}).$$
Thus we have $L(\g)>L(\ti\g_s|_{[t_s,u_s]})$. We define
$\g_s=\ti\g_s|_{[t_s,u_s]}\circ \s_s$ for some smooth change of
parameters $\s_s:[0,1]\to [t_s,u_s]$. The desired result follows.
\end{proof}

\begin{lemma} Let $H\subset M$ and $V$ be as in Lemma 2.1. If $\nu_H\ge \fr{m+k-1}2$, then
$\pi_1(M-V,\p \bar V)=0$.
\end{lemma}
\begin{proof} It suffices to prove that, for any curve $\g : (I, \partial I)\to (M-V, \partial \bar V)$,
there is a homotopy $\ta_s: (I, \partial I)\to (M-V,
\partial \bar V)$ for all $s\in [0,1]$ with $\ta_0=\g$ and $\ta_1(I)\subset \p \bar V$, where $I=[0,1]$.

We argue by contradiction. Assume there is a nonempty set $S$ consisting of the
continuous curves $\g$ as above which can not homotopy to curves in $\p \bar V$.
Let $a$ denote the infimum of the length $L(\g)$ for $\g\in S$.

Let $U$ be a $\de$-tubular neighborhood of $\p \bar V$ so that $U$
strongly deformation retracts to $\partial \bar V$ for some small
$\de>0$. Note that $a\ge 2\de>0$. Consider a sequence of curves
$\s_\ell\in S$ satisfying $L(\s_\ell)\to a$. We may assume further
that $\s_\ell$ is parameterized proportionally to the arc length.
Since $L(\s_\ell)$ is uniformly bounded, the curves $\s_\ell$ form
an equicontinuous sequence. By the Ascoli-Arzela Theorem, passing to
a subsequence, if necessary, we may assume that $\s_\ell$ converges
uniformly to a continuous curve $\s:[0,1]\to M-V$ with length
$L(\s)\le a$.

We claim that $\s\in S$. In fact,  we may easily find normal open
$r_i$-balls $B_1,\cdots,B_{j}$ with $r_i<\de$ for all $i$ and a
partition $t_0=0<t_1<\cdots<t_j=1$ such that
$\s([t_{i-1},t_i])\subset B_i$, for all $i$.  For sufficiently large
$\ell$, there exists a homotopy $\s^u$ between $\s_\ell$ and $\s$
with $\s^u([0,1])\subset \Omega=(M-V)\cup U$. Since $U$ strongly
deformation retracts to $\partial \bar  V$, hence $\s$ is homotopic
to $\s_\ell$, and so $\s \in S$. As a consequence, $L(\s)=a$.

 By the minimality of $L(\s)$
and the first variation formula we know that $\s$ is a geodesic
satisfying that $\s'(0),\s'(1)\perp \p \bar V$. Hence, by using Lemma 2.1 we
obtain a homotopy $\ta$ of $\s$ with length $L(\ta)<L(\s )=a$. A contradiction.
This completes the proof.
\end{proof}

\noindent {\bf Remark 2.3:} By the proofs of Lemmas 2.1 and 2.2 we
know that, if $H\subset M^m$ is closed orientable hypersurface with
orientable normal bundle where $M^m$ possesses positive sectional
curvature, then $H$ separates $M^m$, provided either of the
following conditions holds:

(2.3.1) $\nu_H\ge m/2$;

(2.3.2) $H$ has nonpositive extrinsic curvature and $m\ge 4$.

\section{\bf Proofs of Theorems 1.2 and 1.4}

In the proof we need the following two results from Theorems 0.7,
0.8 and Theorem C in [FMR].

\begin{theorem}[FMR] Let $M$ be an $m$-dimensional closed Riemannian manifold of positive $k$-Ricci
curvature, and $N,H$ closed embedded submanifolds of $M$ with asymptotic indices $\nu _N$, $\nu _H$ respectively.
If $N$ and $H$ intersects transversely, then the following natural homomorphisms $$
i_1:\pi_i(N,N\cap H)\mapsto \pi_i(M,H), \ \ \ \ \ \ \
 i_2:\pi_i(H,N\cap H)\mapsto \pi_i(M,N)$$
 are isomorphisms for $i \le \nu_N+\nu_H-m-k+1$ and are
surjections for $i = \nu_N+\nu_H-m-k+2$.
\end{theorem}

\begin{theorem}[FMR] Let $M$ be an $m$-dimensional closed
Riemannian manifold of positive $k$-Ricci curvature, and let $N$ be
a closed embedded submanifold. Then the inclusion $i : N\mapsto M$
is $(2\nu_N-m-k+2)$-connected.
\end{theorem}

\begin{lemma}
Let $M$ be an $m$-dimensional closed Riemannian manifold of positive
$k$-Ricci curvature. Let $H\subset M$ be a codimension $2$
submanifold with asymptotic index $\nu _H$, and let $N\subset M$ be
a  submanifold of dimension $n$ with asymptotic index $\nu _N$,
which in general position with $H$ (i.e., intersects transversely) ,
then the homomorphism
$$j_*:\pi_1(N-(N\cap H))\to \pi_1(M-H)$$
induced by the inclusion is surjective, provided $\nu_N, \nu_H\ge \fr{m+k-1}{2}$ and
$\nu_N+\nu_H\ge m+k$.
\end{lemma}

\begin{proof}  By Theorem 3.2 we know that $\pi_1(M,N)=0$. By Theorem
3.1 we get further that  $\pi_1(H,N\cap H)=0$.

 Let $V$ be the open
$\varepsilon$-tubular neighborhood of $H$ defined in Lemma 2.1.
Since $\nu_H\ge \fr{m+k-1}{2}$, by Lemma 2.2 it follows that
$$\pi _1(\partial \bar V)\to \pi _1(M-V)$$ is surjective. Observe
that $N\cap \p \bar V$ is contained in $N-(N\cap H)$, since $H$ and
$N$ are in general position. In fact,  $N\cap \p \bar V$ is
diffeomorphic to the normal sphere bundle of $N\cap H$ in $N$.
Therefore, it suffices to show that the inclusion induces an
epimorphism, $\pi_1(N\cap \p \bar V)\mapsto \pi_1(\p \bar V)$ ,
because the inclusion factors through $N-(N\cap H)\to M-V$.

Note that $\partial \bar V$ is an $S^1$-bundle over $H$, whose
pullback on $N\cap H$ is isomorphic to the normal circle bundle of
$N\cap H$ in $N$, by the transversality of $N$ and $H$. Since
$\pi_1(H, N\cap H)=0$, by comparing the exact sequences for the
circle bundles $S^1\to
\partial \bar V\cap N\to N\cap H$ (resp. $S^1\to \partial \bar V\to H$)
we know that   $\pi _1(N\cap \partial \bar V)\to \pi _1(\partial \bar V)$
is surjective. The desired result follows.
\end{proof}

\begin{proof} [Proof of Theorem 1.2] By Florit's Theorem in
\cite{Fl}, $\nu_K\ge  n-4$ if $K$  has nonpositive extrinsic
curvature. Thus, $K\to S^n$ is $(n-7)$-connected, by Theorem 3.2.
This implies that the homology group $H_i(K;\Bbb Z)=0$ for all
$0<i\leq n-8$. By Poincar\'e duality we know that $H_i(K;\Bbb Z)=0$
for all $0<i<n-2$, provided $n-8\geq \frac 12(n-2)$, i.e., $n\ge
14$. As a consequence $K$ is a homotopy sphere, and so homeomorphic
to $S^{n-2}$ by Smale's theorem. On the other hand, by Lemma 2.2 we
know that $\pi _1(\p \bar V)\to \pi _1(S^n-K)$ is surjective. Note
that $\p \bar V$ is a circle bundle over $K$, and so $\pi _1(\p \bar
V )\cong \Bbb Z$. This shows that $\pi _1(S^n-K)$ is abelian, and so
$\pi _1(S^n-K)\cong H_1(S^n-K) \cong H^{n-1}(S^n, K)\cong \Bbb Z$.

If $K$ is totally geodesic, we may use Wilking's theorem (cf. \cite{Wi}) instead in the
above argument, to show that $K$ is homeomorphic to $S^{n-2}$ if $n\ge 5$. The rest of the proof
is the same as above. This proves the theorem.
\end{proof}

\begin{proof}[Proof of Theorem 1.4]
We apply Lemma 3.3 with $k=1$. By Florit's theorem (\cite{Fl}) the
asymptotic index $\nu_H\ge m-4$, and $\nu_N\ge 2n-m$. It is easy to
see that (1.4.1) implies the conditions of Lemma 3.3.

If $N, H$ are both totally geodesic we have $\nu _N=\text{dim} (N)$
(resp. $\nu _H=\text{dim} (H)$). Thus the assumption (1.4.2) implies
the desired result by using Lemma 3.3.
\end{proof}


\begin{thebibliography}{10}

\bibitem[BLP]{BLP}  M. Boileau; B. Leeb; J. Porti,
 {\it Geometrization of $3$-dimensional orbifolds},
Ann. Math.  162(2005),  195-290.
\bibitem[Bo] {Bo} Borisenko, A. A., {\it Certain classes of
multidimensional surfaces}, J. Soviet. Math. 51, no. 2 (1990),
2191--2197; transl. from Ukrainskii Geom. Sb. 29 (1986), 5--12.
\bibitem[BR] {BR} Borisenko, A., Rovenski, V.,
{\it About topology of saddle submanifolds}, Diff. Geom. and its
Applications, 25 (2007), 220--233.
\bibitem[BRT] {BRT} A. Borisenko, M. L. Rabelo, K.
Tenenblat, On saddle submanifolds of Riemannian manifolds, Geom.
Dedicata, 67 (1997), 233--243.
\bibitem[Fl]{Fl} Florit, L., {\it On submanifolds with nonpositive
extrinsic curvature}, Math. Ann. 298 (1994), 187--192.
\bibitem[Fr1]{Fr1} Frankel, T., {\it Manifolds of positive curvature},
Pacific J. Math., 11 (1961), 165--174.
\bibitem[Fr2]{Fr2} Frankel, T., {\it On the fundamental group of a
compact minimal submanifold}, Ann. Math., 83 (1966), 68--73.
\bibitem[FL] {FL} Fulton, W., Lazarsfeld R., {\it Connectivity and
its applications in algebraic geometry}, Lect. Notes in Math., 862,
Springer-Verlag, 26--92.
\bibitem[FM] {FM} Fang, F., Mendon\c ca, S., {\it Complex
immersions in K\"ahler manifolds of positive holomorphic k-Ricci
curvature}, Trans. Amer. Math. Soc. 357 (2005), no. 9, 3725--3738.
\bibitem[FMR]{FMR} Fang, F., Mendon\c ca, S., Rong, X., {\it A
connectedness principle in the geometry of positive curvature},
Comm. Anal. Geom., 13 (2005), no. 4, 671--695.
\bibitem[Go]{Go} Gordon, C. McA., {\it  On the higher-dimensional Smith
conjecture},  Proc. London Math. Soc. 29 (1974), 98--110.
\bibitem[Lf] {Lf} Lefschetz, S. {\it L'analysis situs et la
g\'eom\'etrie alg\'ebrique}, Gauthier-Villars, Paris (1924).
\bibitem[Lv] {Le} Levine, J., {\it
An algebraic classification of some knots of codimension two},
Comment. Math. Helv. 45 (1970), 185--198.
\bibitem[KX] {KX} Kenmotsu K., Xia C.,
{\it Hadamard-Frankel type theorems for manifolds with partially
positive curvature}, Pacific J. Math., 176 (1996), no. 4, 129--139.
\bibitem[MB] {MB} {\it The Smith conjecture},
Papers presented at the symposium held at Columbia University, New
York, 1979. Edited by John W. Morgan and Hyman Bass, Pure and
Applied Mathematics, Academic Press, 1984.
\bibitem[Re]{Re} Reznikov, A., {\it Knotted totally geodesic submanifolds in positively curved
spheres,} Mat. Fiz. Anal. Geom. 7 (2000), 458--463.
\bibitem[Ro]{Ro} Rovenski, V., {\it On the role of partial Ricci
curvature in the geometry of submanifolds and foliations}, Ann.
Polonici Math. 28 no. 1 (1998), 61--82.
\bibitem[Wi]{Wi} Wilking, B., {\it Torus
actions on manifolds of positive sectional curvature}, Acta Math.
191 (2003), 259--297.
\bibitem[Wu]{Wu} Wu, H., {\it Manifolds of partially
positive curvature}, Indiana Univ. Math. J. 36 (1987), 525--548.




\end{thebibliography}
\end{document}